\documentclass{article}
\usepackage{authblk}

\usepackage{graphicx}

\usepackage{amsmath, amssymb, amsthm}
\usepackage{mathdots}
\usepackage{tikz-cd}
\usepackage{tikz, tikz-3dplot}
\usetikzlibrary{patterns, fit}
\usetikzlibrary{arrows}

\usepackage[colorinlistoftodos]{todonotes}

\newtheorem{theorem}{Theorem}[section]
\newtheorem{lemma}[theorem]{Lemma}

\theoremstyle{definition}
\newtheorem{definition}[theorem]{Definition}

\numberwithin{equation}{section}

\usepackage{thm-restate}

\usepackage[e]{esvect}
\newcommand{\Af}[1]{\vv*{A}{#1}}

\newcommand{\intv}{\mathbb{I}}

\newcommand{\vect}{\mathrm{vect}}

\newcommand{\smat}[1]{ \left[\begin{smallmatrix} #1 \end{smallmatrix}\right] }
\newcommand{\suchthat}{\ifnum\currentgrouptype=16 \;\middle|\;\else\mid\fi}

\newcommand{\fun}[2]{\left[#1,#2\right]}

\newcommand{\lessf}{\mathbin{{<}_{\scriptscriptstyle f}}}

\usepackage{mathtools}
\newcommand{\defeq}{\coloneqq}

\DeclareMathOperator{\rep}{rep}

\DeclareMathOperator{\End}{End}

\DeclareMathOperator{\Hom}{Hom}

\makeatletter
\def\Int #1{\expandafter\Int@i#1\@nil}
\def\Int@i #1,#2\@nil{{#1{:}#2}}

\def\IntC #1{\expandafter\IntC@i#1\@nil}
\def\IntC@i #1,#2\@nil{{#1{,}#2}}

\def\itoi #1{\expandafter\itoi@i#1\@nil}
\def\itoi@i #1,#2,#3,#4\@nil{_{\Int{#1,#2}}^{\Int{#3,#4}}}
\makeatother

\def\clap#1{\hbox to 0pt{\hss#1\hss}}

\def\mathrlap{\mathpalette\mathrlapinternal}

\def\mathrlapinternal#1#2{%
  \rlap{$\mathsurround=0pt#1{#2}$}}

\makeatletter
\let\save@mathaccent\mathaccent
\newcommand*\if@single[3]{%
  \setbox0\hbox{${\mathaccent"0362{#1}}^H$}%
  \setbox2\hbox{${\mathaccent"0362{\kern0pt#1}}^H$}%
  \ifdim\ht0=\ht2 #3\else #2\fi
  }
\newcommand*\rel@kern[1]{\kern#1\dimexpr\macc@kerna}
\newcommand*\widebar[1]{\@ifnextchar^{{\wide@bar{#1}{0}}}{\wide@bar{#1}{1}}}
\newcommand*\wide@bar[2]{\if@single{#1}{\wide@bar@{#1}{#2}{1}}{\wide@bar@{#1}{#2}{2}}}
\newcommand*\wide@bar@[3]{%
  \begingroup
  \def\mathaccent##1##2{%
    \let\mathaccent\save@mathaccent
    \if#32 \let\macc@nucleus\first@char \fi
    \setbox\z@\hbox{$\macc@style{\macc@nucleus}_{}$}%
    \setbox\tw@\hbox{$\macc@style{\macc@nucleus}{}_{}$}%
    \dimen@\wd\tw@
    \advance\dimen@-\wd\z@
    \divide\dimen@ 3
    \@tempdima\wd\tw@
    \advance\@tempdima-\scriptspace
    \divide\@tempdima 10
    \advance\dimen@-\@tempdima
    \ifdim\dimen@>\z@ \dimen@0pt\fi
    \rel@kern{0.6}\kern-\dimen@
    \if#31
      \overline{\rel@kern{-0.6}\kern\dimen@\macc@nucleus\rel@kern{0.4}\kern\dimen@}%
      \advance\dimen@0.4\dimexpr\macc@kerna
      \let\final@kern#2%
      \ifdim\dimen@<\z@ \let\final@kern1\fi
      \if\final@kern1 \kern-\dimen@\fi
    \else
      \overline{\rel@kern{-0.6}\kern\dimen@#1}%
    \fi
  }%
  \macc@depth\@ne
  \let\math@bgroup\@empty \let\math@egroup\macc@set@skewchar
  \mathsurround\z@ \frozen@everymath{\mathgroup\macc@group\relax}%
  \macc@set@skewchar\relax
  \let\mathaccentV\macc@nested@a
  \if#31
    \macc@nested@a\relax111{#1}%
  \else
    \def\gobble@till@marker##1\endmarker{}%
    \futurelet\first@char\gobble@till@marker#1\endmarker
    \ifcat\noexpand\first@char A\else
      \def\first@char{}%
    \fi
    \macc@nested@a\relax111{\first@char}%
  \fi
  \endgroup
}
\makeatother

\title{The Whole in the Parts: Putting $n$D Persistence Modules Inside Indecomposable $(n+1)$D Ones}

\author[1]{Micka\"{e}l Buchet}
\affil[1]{Institute for Geometry, Graz University of Technology, Austria}
\affil[1]{buchet@tugraz.at}
\author[2]{Emerson G. Escolar}
\affil[2]{Graduate School of Human Development and Environment, Kobe University, Japan; and Center for Advanced Intelligence Project, RIKEN, Tokyo, Japan}
\affil[2]{e.g.escolar@people.kobe-u.ac.jp}

\begin{document}
\maketitle
\begin{abstract}
Multidimensional persistence has been proposed to study the persistence of topological features in data indexed by multiple parameters.
In this work, we further explore its algebraic complications from the point of view of higher dimensional indecomposable persistence modules containing lower dimensional ones as hyperplane restrictions. 
Our previous work constructively showed that any finite rectangle-decomposable $n$D persistence module is the hyperplane restriction of some indecomposable $(n+1)$D persistence module, as a corollary of the result for $n=1$. Here, we extend this by dropping the requirement of rectangle-decomposability. Furthermore, in the case that the underlying field is countable, we construct an indecomposable $(n+1)$D persistence module containing all $n$D persistence modules, up to isomorphism, as hyperplane restrictions. Finally, in the case $n=1$, we present a minimal construction that improves our previous construction.
\end{abstract}

\noindent\textbf{Keywords} Multidimensional Persistent Homology, Quiver Representation Theory, Lower Bounds\\
\textbf{Mathematics Subject Classification (2010)} 16G20, 55N99

\section{Introduction}

The structure of multidimensional persistence modules is known to be very complicated, with no complete discrete invariant \cite{carlsson2009theory}.
In the $2$D case, one idea for extracting information is to restrict them along lines in order to obtain $1$D persistence modules, which have structures completely described by barcodes.
This idea arises naturally in many situations.
For example, this is used in the software RIVET~\cite{lesnick2015interactive} to provide an interactive graphical exploration of $2$D persistence modules via barcodes of line restrictions.
In fact, the rank invariant \cite{carlsson2009theory} is equivalent to the barcodes of all such line restrictions of
the $2$D persistence module.
Furthermore, the idea of line restrictions is also the foundation of the kernel for multidimensional persistence proposed in~\cite{corbet2019kernel}.
The notion naturally generalizes to higher dimensions, where we can consider restrictions of $m$D persistence modules to $n$D hyperplanes.
We focus on hyperplane restrictions where $m=n+1$.

In a previous work~\cite{buchet2020every} we
showed that every finite dimensional $1$D persistence module can in fact be found as a line restriction of some indecomposable $2$D persistence module.
Moreover we built a pointwise finite dimensional indecomposable $2$D persistence module that contains all finite dimensional $1$D persistence modules as line restrictions.
This implies that all the ``complexity'' of $1$D persistence modules is contained in the indecomposable $2$D persistence modules.
We also extended the result,
showing that every finite rectangle-decomposable $n$D persistence module is the restriction of some indecomposable $(n+1)$D persistence module.
In this paper, we address several questions left open by the previous work.

First, in Theorem~\ref{thm:nD} we extend the $n$D result to the general case by providing a construction that, for any $n$D persistence module $V$, builds an indecomposable $(n+1)$ persistence module $M$ having $V$ as a hyperplane restriction. That is, we drop the requirement of rectangle-decomposability.
Note that a similar but slightly more space-consuming construction has recently and independently been proposed by Moore~\cite{moore2020hyperplane}.
Next, in Theorem~\ref{thm:lotr}, we build a single indecomposable $(n+1)$D persistence module containing all finite dimensional $n$D persistence modules as hyperplane restrictions, when the base field is countable.

Another question left open was the minimality of the construction.
The construction previously proposed, as well as its extension in the first part of this paper, requires the stacking of several layers of $n$D persistence modules to construct an indecomposable $(n+1)$D persistence module.
What is the minimum number of layers needed?
In Theorem~\ref{thm:minimal}  we provide an improved construction requiring exactly three layers for $1$D persistence modules and show that two layers are not enough.
Finally, in Theorem~\ref{thm:minimal_rect}, we show that the construction using three layers can be generalized to realize finite rectangle-decomposable $n$D persistence modules as hyperplane restrictions of indecomposable $(n+1)$D persistence modules. Furthermore, using a fourth layer, we can realize any $n$D persistence module as a hyperplane restriction of an indecomposable $(n+1)$D persistence module. We do not know whether or not this is not the minimum number of layers needed when $n > 1$.

\section{Definitions}
We denote elements (also called vertices) of $\mathbb{Z}^n$ as $n$-tuples $\vv{b} = (b^1,b^2,\hdots,b^n)\in\mathbb{Z}^n$, and consider the poset $(\mathbb{Z}^n,\leq)$, where $\vv{b} \leq \vv{d}$ if and only if $b^k \leq d^k$ for all $k=1,\hdots,n$. We view the poset $(\mathbb{Z}^n,\leq)$ as a category in the usual way.
For $\vv x \leq \vv y$ in $\mathbb{Z}^n$, we write the unique morphism from $\vv x$ to $\vv y$ as $\vv x \leq \vv y$, so composition is written as $(\vv y \leq \vv z)(\vv x \leq \vv y) = (\vv x \leq \vv z)$.

A pointwise finite dimensional \emph{$n$D persistence module} is a functor $M$ from $\mathbb{Z}^n$ to the category of finite dimensional $K$-vector spaces, written as $M \in \rep \mathbb{Z}^n = \fun{\mathbb{Z}^n}{\vect_K}$. By functoriality, $M(\vv y \leq \vv z)M (\vv x\leq \vv y) = M(\vv x \leq \vv z)$ for all $\vv x \leq \vv y \leq \vv z$ in $\mathbb{Z}^n$. We shall use the term ``commutativity relation'' to denote this requirement, when talking about persistence modules. In this context, we also call $\mathbb{Z}^n$ an $n$-dimensional commutative grid. The maps $M(\vv x \leq \vv y)$ are also called the internal maps of $M$.

The \emph{support} of a persistence module $M$ is
$
  \{M(\vv{x}) \neq 0 \suchthat \vv{x} \in \mathbb{Z}^n\};
$
$M$ is said to have finite support if its support is a finite set.
Throughout this work, unless specified otherwise, \emph{$n$D persistence module} shall be taken to mean pointwise finite dimensional $n$D persistence module with finite support.

As objects of a category of functors to $\vect_K$, we have the usual concepts of morphisms (natural transformations) between $n$D persistence modules, and direct sums, decomposability, and indecomposability of $n$D persistence modules.
Furthermore, the endomorphism ring of an $n$D persistence module $M$ is $\End(M) = \Hom(M,M)$, the set of morphisms from $M$ to itself, with multiplication by composition.
Since we want to construct indecomposable persistence modules, the follow fact is useful, where we recall that a ring with unity $R$ is said to be \emph{local} if $0 \neq 1$ in $R$ and for each $x\in R$, $x$ or $1-x$ is invertible.
\begin{lemma}[Corollary 4.8 of \cite{assem2006elements}]\label{lem:locality}
  \label{lem:indec_endo}
  Let $V$ be a representation of a finite bound quiver $Q$.
  \begin{enumerate}
  \item If $\End{V}$ is local then $V$ is indecomposable.
  \item If $V$ is finite dimensional and indecomposable, then $\End{V}$ is local.
  \end{enumerate}
\end{lemma}
The above corollary is stated for finite bound quivers. However, we note that for each $n$D persistence module $M$ with finite support, we can consider $M$ as a representation of the quiver bound by commutative relations corresponding to the support of $M$. This leads to the fact that for $M$ an $n$D persistence module with finite support, $\End(M)$ is local if and only if $M$ is indecomposable. Note that $M$ is automatically finite dimensional if it has finite support, since $M \in \fun{\mathbb{Z}^n}{\vect_K}$ is pointwise finite dimensional.

Next, we recall the definition of hyperplane restrictions.
\begin{definition}[Hyperplane Restriction, Definition~7.6 of \cite{buchet2020every}]
  \label{defn:hyperplane_restriction}
  Let $n$ be a positive integer.
  \begin{enumerate}
  \item A \emph{hyperplane} $L$ is a functor $L : (\mathbb{Z}^n,\leq) \rightarrow (\mathbb{Z}^{n+1},\leq)$ that is injective on objects.
  \item Let $V$ be an $n$D persistence module. We say that $V$ is a \emph{hyperplane restriction} of the $(n+1)$D persistence module $M$ if there is a hyperplane $L$ such that $ML \cong V$.
  \end{enumerate}
\end{definition}
In the case $n=1$, we have a \emph{line} defined as a functor $L : (\mathbb{Z},\leq) \rightarrow (\mathbb{Z}^{2},\leq)$ that is injective on objects, and correspondingly a \emph{line restriction} to $1$D persistence modules.
Note that the points in the image of a line (hyperplane) $L$ do not need to be ``adjacent''. That is, we are free to skip points in taking line restrictions (hyperplane restrictions). We do not take advantage of this in Theorem~\ref{thm:rectangles}, Theorem~\ref{thm:nD}, nor Theorem~\ref{thm:lotr}. On the other hand, our compressed constructions in Theorem~\ref{thm:minimal} and Theorem~\ref{thm:minimal_rect} take advantage of this.

Let us also recall the following definition.
For $\vv{b} \leq \vv{d}$ in $\mathbb{Z}^n$, the \emph{finite rectangle $n$D persistence module} $\intv[\vv{b},\vv{d}]$ is the functor $M$ that associates $M(\vv{x}) = K$ for $\vv{b} \leq \vv{x} \leq \vv{d}$,
and $M(\vv{x} \leq \vv{y})$ the identity map if $\vv{b} \leq \vv{x} \leq \vv{y} \leq \vv{d}$ and zero otherwise.
Finally, an $n$D persistence module is said to be \emph{finite rectangle-decomposable} if it is isomorphic to a finite direct sum of finite rectangle $n$D persistence modules.
In the $1$D case, rectangles $\intv[b,d]$ ($b\leq d\in\mathbb{Z}$) are called \emph{intervals}. Every pointwise finite dimensional $1$D persistence module is interval-decomposable (see \cite[Theorem~2.8]{chazal2016structure} for the result over $\mathbb{Z}$, and over $\mathbb{R}$, \cite{crawley2015decomposition}).

A $(n+1)$D persistence module can be interpreted as a sequence of $n$D persistence modules linked by morphisms.
Conversely, given a sequence $(M_i)_{i\in \mathbb{Z}}$ and morphisms $\phi_i:M_i\rightarrow M_{i+1}$, the \emph{stacking} of the modules and morphisms is defined as the $(n+1)$D persistence module $M$ such that for $\vv x \leq \vv y \in \mathbb{Z}^n$ and $i \in \mathbb{Z}$ we define
$M(\vv x,i)=M_i(\vv x)$,
$M((\vv x,i)\leq(\vv y, i))=M_i(\vv x\leq\vv y)$
and
$M((\vv x, i)\leq(\vv x, i+1))=\phi_i(\vv x)$,
where $(\vv x, i)$ is viewed as an element of $\mathbb{Z}^{n+1}$ in the obvious way.
The remaining internal maps of $M$ are defined through composition.
Note that they are well-defined due to the commutativity conditions in the $M_i$ and the fact that all $\phi_i$ are morphisms.
See the Stacking Lemma~\cite[Lemma~3.1]{buchet2020every}, for the $1$D to $2$D case.
Given an $M$ built by stacking $(M_i)$, we call each $M_i$ a layer of $M$.
We say that $M$ has $l$ layers when exactly $l$ of the $M_i$ are non-zero.
It is clear that if $M$ is indecomposable, then all the non-zero layers are adjacent.

\section{Results}
First, let us recall the following theorem and give a quick sketch of its proof.
\begin{theorem}[{\cite[Theorem~7.7]{buchet2020every}}]
  \label{thm:rectangles}
  Let $V$ be a finite rectangle-decomposable $n$D persistence module. Then, there exists an indecomposable $(n+1)$D persistence module $M$ with finite support such that $V$ is a hyperplane restriction of $M$.
\end{theorem}
\begin{proof}[Sketch of proof]
  Without loss of generality, we assume that $V = \bigoplus_{i=1}^m \intv[\vv{b_i},\vv{d_i}]$.
  Below, we review the construction in \cite{buchet2020every}, which consists of the steps ``separate-and-shift'', ``verticalization'', and ``coning''.
  \begin{itemize}
  \item We ``separate-and-shift'' to distinct $\vv{d'_i} \in \mathbb{Z}^n$ such that $\vv{d'_i} \geq \vv{d_i}$ for all $i$ and $\frac{\vv{b_j} + \vv{d'_j}}{2} \leq \vv{d'_i}$ for all pairs $i$ and $j$.

  \item Take
    $
    \vv{\mu} = \max\limits_{i}\left(\frac{\vv{b_i} + \vv{d'_i}}{2}\right)
    $
    where $\max$ is taken component-wise. As part of ``verticalization'', we define $\vv{b_i'} = 2\vv{\mu} - \vv{d'_i}$, and check that
    $
    \vv{b_i} \leq \vv{b'_i} \leq \vv{\mu} \leq \vv{d'_j}
    $
    for all pairs $i$, $j$.

  \item Finally, ``coning'' involves the finite rectangle
    $
      I_V \defeq \intv\left[\max\limits_i\left(\vv{b'_i}\right), \max\limits_j\left(\vv{d'_j}\right)\right].
    $
  \end{itemize}

  We get the object $S$,
  an ($n$D persistence module)-valued representation of the quiver
  $
    \Af{4}: 0 \rightarrow 1 \rightarrow 2 \rightarrow 3,
  $
  given by
  \begin{equation}
    \label{eq:rectS}
    S:
    \begin{tikzcd}[ampersand replacement=\&, column sep=large]
      I_V \rar{\smat{1\\\vdots\\1}} \&
      \widebar{V} \rar{\smat{1 & & \\ & \ddots & \\ & & 1}} \&
      V' \rar{\smat{1 & & \\ & \ddots & \\ & & 1}} \&
      V
    \end{tikzcd}
  \end{equation}
  where
  \begin{itemize}
  \item $S(3)\defeq V$,
  \item $S(2) \defeq V'  = \bigoplus_{i=1}^m \intv[\vv{b_i},\vv{d'_i}]$ is obtained by separate-and-shift,
  \item $S(1) \defeq \widebar{V} = \bigoplus_{i=1}^m \intv[\vv{b'_i},\vv{d'_i}]$ by verticalization (of $V'$), and
  \item $S(0) \defeq I_V$ by coning.
  \end{itemize}
  The morphisms between the entries of $S$ in Diagram~\eqref{eq:rectS} are given in the ``matrix formalism'', as explained in \cite{buchet2020every}. We warn that the ``$1$'''s in Diagram~\eqref{eq:rectS} do not represent identity maps, but rather the identity $1\in K$ multiplied by a chosen basis element for the homomorphism space between the corresponding rectangle summands, as defined in Lemma~7.3 of \cite{buchet2020every}. These chosen basis elements shall be called \emph{canonical homomorphisms} in the rest of this paper.

  The arguments in \cite{buchet2020every} show that every endomorphism of $S$ is a multiplication by some scalar $c \in K$. This shows that $\End(S) \cong K$, and thus $S$ is indecomposable.
  By stacking its entries and morphisms, $S$ can be viewed as 
  an $(n+1)$D persistence module. The $(n+1)$D persistence module corresponding to $S$ has finite support, and $V$ is a hyperplane restriction by a natural inclusion $\mathbb{Z}^n \hookrightarrow \mathbb{Z}^{n+1}$.
\end{proof}

Let us extend the above result to $n$D persistence modules in general.
Using the existence of projective covers and injective envelopes (see Section~I.5 of \cite{assem2006elements}) for $n$D persistence modules over finite commutative grids, together with the following Lemma, we are able to use the finite rectangle-decomposable modules as an intermediate step to build corresponding indecomposables in the general case. Note that directly using projective covers  over the entire  $\mathbb{Z}^n$ leads to rectangles with infinite support, something we want to avoid.

\begin{lemma}[Padding by zeros]
\label{lem:padding}
    Let $G$ be the subcategory $G = \{-m,\hdots,m\}^n \subset \mathbb{Z}^n$. Then for each functor
    $V' \in \fun{G}{\vect_K}$,
    \[
        V(x) = 
        \left\{
        \begin{array}{ll}
        V'(x) & \text{if } x \in G, \\
        0 & \text{otherwise,}
        \end{array}
        \right.
        V(x\leq y) = 
        \left\{
        \begin{array}{ll}
        V'(x\leq y) & \text{if } x \in G \text{ and } y \in G, \\
        0 & \text{otherwise,}
        \end{array}
        \right.
    \]
    defines an $n$D persistence module $V$. 
    For each natural transformation $f': V'\rightarrow W'$ in $\fun{G}{\vect_K}$,
    \[
        f_x = 
        \left\{
        \begin{array}{ll}
        f'_x & x \in G \\
        0 & \text{otherwise.}
        \end{array}
        \right.
    \]
    defines a morphism $f: V\rightarrow W$, where the $n$D persistence modules $V$ and $W$ are defined from $V'$ and $W'$ as above.
    Moreover, if $f'$ is an injection or a surjection, then so is $f$.
\end{lemma}
\begin{proof}
We check the functoriality of $V$. It is clear that $V(x\leq x)$ is the identity for all $x\in \mathbb{Z}^n$.
To see that 
$
    V(y \leq z) V(x \leq y) = V(x \leq z)
$
for all $x,y,z,\in \mathbb{Z}^n$,
we note that if either $x$ or $z$ is not in $G$, then both sides of the equation are $0$. If $y$ is not in $G$, then at least one of $x$ or $z$ must also be not in $G$, since $G$ is convex. Thus, if at least one of $x$, $y$, or $z$ is not in $G$, the equality holds. Otherwise, the equality follows from the functoriality of $V'$.

To check that $f$ is a natural transformation, we need to verify the commutativity of
\[
\begin{tikzcd}
  V(x) \rar{V(x\leq y)}\dar{f_x} & V(y) \dar{f_y} \\
  W(x) \rar{W(x\leq y)} & W(y)\mathrlap{.}
\end{tikzcd}
\]
If both $x$ and $y$ are in $G$, then the diagram restricts to one involving $V'$, $W'$, and $f'$, and commutativity follows from the fact that $f'$ is a natural transformation. If $x$ or $y$ is not in $G$, commutativity is immediate as either $V(x)=0$ and $W(x) = 0$, or $V(y) = 0$ and $W(y)=0$.

The fact that $f$ is an injection (a surjection, respectively) if $f'$ is an injection (a surjection, respectively) follows immediately from the definition.
\end{proof}

\begin{theorem}
  \label{thm:nD}
  Let $V$ be an $n$D persistence module with finite support. Then, there exists an indecomposable $(n+1)$D persistence module $M$ with finite support such that $V$ is a hyperplane restriction of $M$.
\end{theorem}

\begin{proof}
  Let $V$ be as given. We observe that there exists a finite rectangle-decomposable $n$D persistence module $R$ and a surjection $p:R \twoheadrightarrow V$, by the following.
  Since $V$ has finite support, there exists a finite subgrid $G = \{-m,\hdots,m\}^n \subset \mathbb{Z}^n$ containing the support of $V$. 
  Restrict $V$ to $(G,\leq)$ as $V|_G$.
  Let $R'$ be the projective cover of $V|_G$, and thus we have a surjection $p':R'\twoheadrightarrow V|_G$ in $\fun{G}{\vect_K}$.
  Using Lemma~\ref{lem:padding} to pad $R'$ with zero spaces and zero maps, we obtain the persistence module $R$ on $\mathbb{Z}^n$. Note that $R$ is finite rectangle-decomposable, due to the form of projective representations of the finite commutative grid $(G,\leq)$. Similarly, from $p'$ we obtain the surjection $p: R\twoheadrightarrow V$, as required.
  

  Then, using the construction in the proof of Theorem~\ref{thm:rectangles} applied to 
  $R$, we obtain the representation
  \[
    S:
    \begin{tikzcd}[ampersand replacement=\&, column sep=large]
      I_R \rar{\smat{1\\\vdots\\1}} \&
      \widebar{R} \rar{\smat{1 & & \\ & \ddots & \\ & & 1}} \&
      R' \rar{\smat{1 & & \\ & \ddots & \\ & & 1}} \&
      R.
    \end{tikzcd}
  \]
  We remind the reader that  the ``$1$'''s are not identity maps, but rather the identity $1\in K$ multiplied by the canonical homomorphisms between the corresponding rectangle summands.
  Appending $V$, we obtain %
  \[
    S':
    \begin{tikzcd}[ampersand replacement=\&, column sep=large]
      I_R \rar{\smat{1\\\vdots\\1}} \&
      \widebar{R} \rar{\smat{1 & & \\ & \ddots & \\ & & 1}} \&
      R' \rar{\smat{1 & & \\ & \ddots & \\ & & 1}} \&
      R \arrow[two heads]{r}{p}
      \& V.
    \end{tikzcd}
  \]
  Let us compute $\End(S')$. For each $\phi \in \End(S')$, $\phi$ restricted to the first four entries is an endomorphism of $S$. We recall that it was shown in \cite{buchet2020every} that $\End(S) \cong K$. Thus, $\phi$ is determined by a constant $c \in K$ and a morphism $\phi_V$ such that the following diagram commutes:
  \[
    \begin{tikzcd}[ampersand replacement=\&]
    R \arrow[two heads]{r}{p} \dar[swap]{c 1_{R}} \& V \dar{\phi_V}\\
    R \arrow[two heads]{r}{p} \& V.
    \end{tikzcd}
  \]
  Let $v\in V$, that is, $v$ is an element in one the vector spaces of $V$.
  By the fact that $p$ is surjective, there exists an $r\in R$ such that $p(r) = v$. Thus,
  \[
    \phi_V(v) = \phi_V(p(r)) = p(c 1_R(r)) = c p(r) = cv,
  \]
  from which we conclude that $\phi_V$ is the multiplication by the scalar $c$. Thus $\End(S') \cong K$, and $S'$ is indecomposable.

  Once again by stacking, $S'$ can be viewed as an $(n+1)$D persistence module. The $(n+1)$D persistence module corresponding to $S'$ has finite support, and $V$ is a hyperplane restriction by a natural inclusion $\mathbb{Z}^n \hookrightarrow \mathbb{Z}^{n+1}$.
\end{proof}

The dual construction works similarly.
After restricting $V$ to the finite commutative grid $(G,\leq)$, we obtain the injective envelope $T'$ together with an injection $j':V|_G\rightarrowtail T'$.
Using Lemma~\ref{lem:padding} to pad $T'$ and $j'$ with zeros, we obtain a finite rectangle-decomposable persistence module $T$ together with an injection $j:V\rightarrowtail T$. We then apply the dual construction of Theorem~\ref{thm:rectangles} and obtain the representation:
\[
  S'':
  \begin{tikzcd}[ampersand replacement=\&, column sep=large]
    V \arrow[tail]{r}{j}\& T \rar{\smat{1 & & \\ & \ddots & \\ & & 1}} \&
    T' \rar{\smat{1 & & \\ & \ddots & \\ & & 1}} \&
    \widebar{T}  \rar{\smat{1 & \hdots & 1}}
    \& I_R.
  \end{tikzcd}
\]
As before, this can be viewed as an $(n+1)$D persistence module via stacking.
By an argument dual to the one above, its endomorphism ring is isomorphic to the field $K$.
Hence $S''$ is also indecomposable.

The combination of the two constructions gives birth to the following indecomposable module $H$, which is of interest for our next construction (Theorem~\ref{thm:lotr}).
\begin{equation}
\label{eq:candywrap}
  \begin{tikzcd}[ampersand replacement=\&, column sep = 2.7em]
    I_R \rar{\smat{1\\\vdots\\1}} \&
    \widebar{R} \rar{\smat{1 & & \\ & \ddots & \\ & & 1}} \&
    R' \rar{\smat{1 & & \\ & \ddots & \\ & & 1}} \&
    R \arrow[two heads]{r}{p} \&
    V \arrow[tail]{r}{j}\& T \rar{\smat{1 & & \\ & \ddots & \\ & & 1}} \&
    T' \rar{\smat{1 & & \\ & \ddots & \\ & & 1}} \&
    \widebar{T}  \rar{\smat{1 & \hdots & 1}}
    \& I_T.
  \end{tikzcd}
\end{equation}
The resulting module $H$ has the properties of ``candy modules'' as defined in~\cite{buchet2020every}. Namely,
\begin{enumerate}
\item the support of $H$ is contained in a hyperrectangle, with the lower right corner and the upper left corner supporting a one dimensional vector space, and
\item the endomorphism ring of $H$ is isomorphic to $K$.
\end{enumerate}

The idea is that $V$ is wrapped in a ``candy wrapper'' with ``handles''. 
The lower right and upper left corners are the handles and prove to be convenient for stringing along with other candies.
While the notion of lower right and upper left corners is obvious for $2$D persistence modules, we need to provide a definition in the $(n+1)$D case for $n > 1$.

First, we set as a privileged coordinate axis the one along which we stack the $n$D persistence modules.
We fix it to be the last coordinate.
The ``upper left corner'' is the vertex of the containing hyperrectangle for whose all coordinates are minimized except the last one which is maximized. On the other hand, the ``lower right corner'' is the vertex where all coordinates are maximized except for the last one which is minimized.
Taking the last coordinate as
height,
these two corners geometrically corresponds to the minimal element of the top facet and the maximal element of the bottom facet, as illustrated in Figure~\ref{fig:corners}.

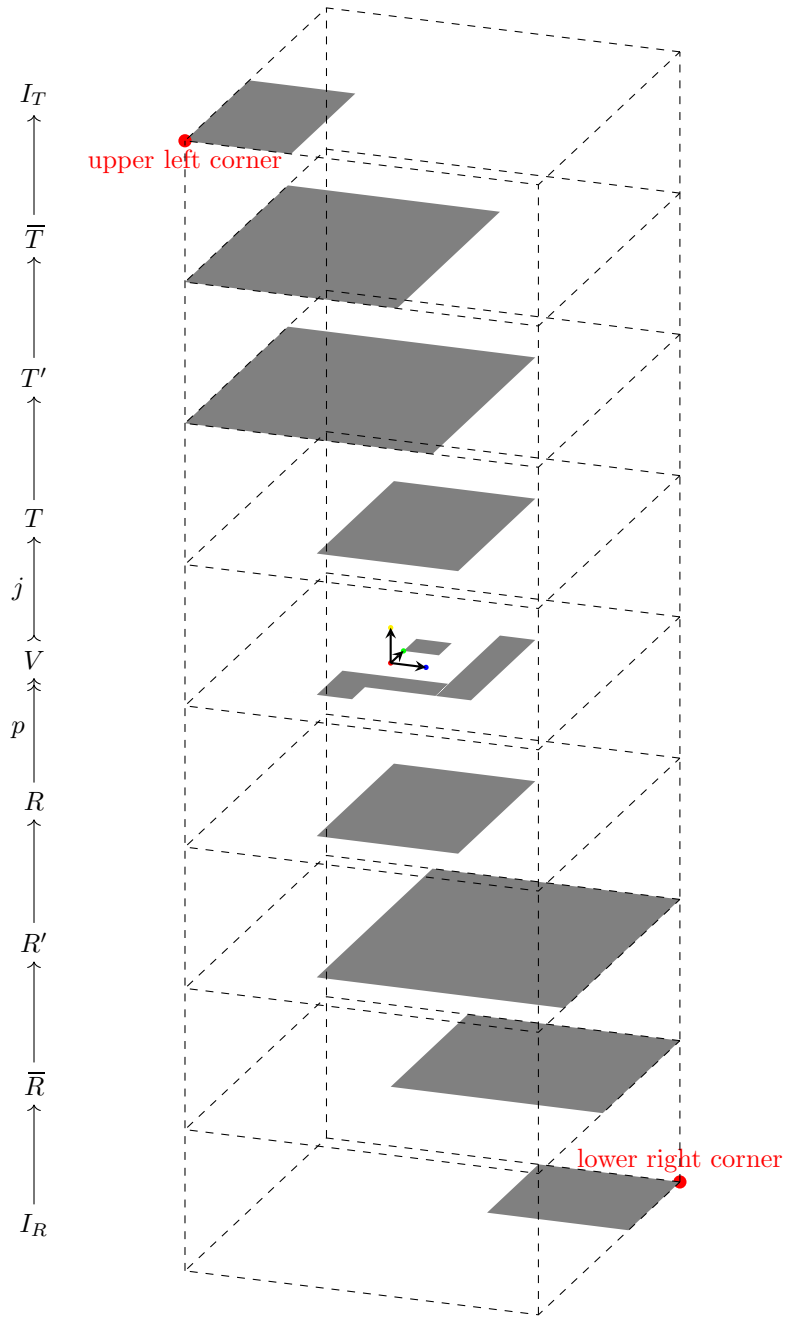
\begin{figure}
\begin{center}
\tdplotsetmaincoords{70}{20}
\begin{tikzpicture}[tdplot_main_coords, scale=.5]
    \fill[gray] (-2,-3,0) -- (-1,-3,0) -- (-1,-1,0) -- (2,-1,0) -- (2,-2,0) -- (3,-2,0) -- (3,3,0) -- (2,3,0) -- (2,-2,0) -- (0,-2) -- (0,-3,0) -- (-2,-3,0);
    \fill[gray] (1,2,0) -- (1,1,0) -- (0,1,0) -- (0,2,0) -- (1,2,0);
    
    \fill[gray] (-1,-3,4) -- (3,-3,4) -- (3,3,4) -- (-1,3,4);
    
    \fill[gray] (-4,-5,8) -- (3,-5,8) -- (3,3,8) -- (-4,3,8);
    
    \fill[gray] (-4,-5,12) -- (-4,3,12) -- (2,3,12) -- (2,-5,12) -- (-4,-5,12);
    
    \fill[gray] (-4,-5,16) -- (-4,0,16) -- (-1,0,16) -- (-1,-5,16) -- (-4,-5,16);
    
    \fill[red] (-4,-5,16) circle (5pt) node[below] {upper left corner};
    \fill[red] (6,6,-16) circle (5pt) node[above] {lower right corner};
    
    \fill[gray] (-1,-3,-4) -- (3,-3,-4) -- (3,3,-4) -- (-1,3,-4);
    
    \fill[gray] (-1,-3,-8) -- (6,-3,-8) -- (6,6,-8) -- (-1,6,-8);
    
    \fill[gray] (0,0,-12) -- (6,0,-12) -- (6,6,-12) -- (0,6,-12);
    
    \fill[gray] (2,2,-16) -- (6,2,-16) -- (6,6,-16) -- (2,6,-16);
    
    \draw[thin,black,dashed] (6,6,-16) -- (6,6,16) (-4,-5,-16) -- (-4,-5,16) (-4,6,-16) -- (-4,6,16) (6,-5,-16) -- (6,-5,16);
    \draw[thin,black,dashed] (6,6,0) -- (-4,6,0) -- (-4,-5,0) -- (6,-5,0) -- (6,6,0);
    \draw[thin,black,dashed] (6,6,4) -- (-4,6,4) -- (-4,-5,4) -- (6,-5,4) -- (6,6,4);
    \draw[thin,black,dashed] (6,6,8) -- (-4,6,8) -- (-4,-5,8) -- (6,-5,8) -- (6,6,8);
    \draw[thin,black,dashed] (6,6,12) -- (-4,6,12) -- (-4,-5,12) -- (6,-5,12) -- (6,6,12);
    \draw[thin,black,dashed] (6,6,16) -- (-4,6,16) -- (-4,-5,16) -- (6,-5,16) -- (6,6,16);
    \draw[thin,black,dashed] (6,6,-4) -- (-4,6,-4) -- (-4,-5,-4) -- (6,-5,-4) -- (6,6,-4);
    \draw[thin,black,dashed] (6,6,-8) -- (-4,6,-8) -- (-4,-5,-8) -- (6,-5,-8) -- (6,6,-8);
    \draw[thin,black,dashed] (6,6,-12) -- (-4,6,-12) -- (-4,-5,-12) -- (6,-5,-12) -- (6,6,-12);
    \draw[thin,black,dashed] (6,6,-16) -- (-4,6,-16) -- (-4,-5,-16) -- (6,-5,-16) -- (6,6,-16);
    
    \fill[red] (0,0,0) circle (2pt);
    \fill[blue] (1,0,0) circle (2pt);
    \fill[green] (0,1,0) circle (2pt);
    \fill[yellow] (0,0,1) circle (2pt);
    \draw[-stealth, black, thick] (0,0,0) -- (0,0,1);
    \draw[-stealth, black, thick] (0,0,0) -- (1,0,0);
    \draw[-stealth, black, thick] (0,0,0) -- (0,1,0);
    
    \node (V) at (-9,-3,0) {$V$};
    \node (T) at (-9,-3,4) {$T$};
    \node (Tp) at (-9,-3,8) {$T'$};
    \node (Tb) at (-9,-3,12) {$\widebar T$};
    \node (It) at (-9,-3,16) {$I_T$};
    \node (R) at (-9,-3,-4) {$R$};
    \node (Rp) at (-9,-3,-8) {$R'$};
    \node (Rb) at (-9,-3,-12) {$\widebar R$};
    \node (Ir) at (-9,-3,-16) {$I_R$};
    
    \draw[>->] (V.north) -- node[left] {$j$} (T.south);
    \draw[->] (T.north) -- (Tp.south);
    \draw[->] (Tp.north) -- (Tb.south);
    \draw[->] (Tb.north) -- (It.south);
    \draw[->>] (R.north) -- node[left] {$p$} (V.south);
    \draw[->] (Ir.north) -- (Rb.south);
    \draw[->] (Rb.north) -- (Rp.south);
    \draw[->] (Rp.north) -- (R.south);
    
\end{tikzpicture}
\caption{Support of a 3D module obtained by our construction with its two corners marked.}\label{fig:corners}
\end{center}
\end{figure}

Given two $n$D candy modules $A$ and $B$, we can define the concatenation operation $A\circ B$ along the last two dimensions as was done for $2$D modules.
First translate the module $B$ so that the lower right corner of $A$ and the upper left corner of $B$ coincide.
Then move $B$ by a translation of $-1$ in the last coordinate and $+1$ in the second-to-last coordinate.
We then add one vertex $x$ supporting the one dimensional vector space $K$ at the immediate left of $B$, id est the vertex with coordinates given by the upper left corner of the translated $B$ plus $-1$ in the second-to-last coordinate.
We then add two arrows supporting isomorphisms from $K$ at $x$ to the one dimensional vector spaces at the lower right corner of $A$ and at the upper left corner of $B$.

Projecting the support of the modules to the hyperplane containing both the lower right corner of $A$ and the upper left corner of $B$, we obtain the picture in Figure~\ref{fig:concat2d}.
Note that the support of $B$ is located beyond this hyperplane while the support of $A$ is in front.
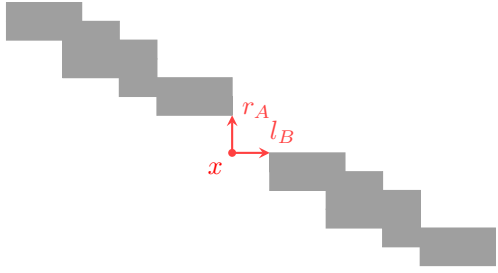
\begin{figure}
  \newcommand{\drawprofile}[8]{
    \def\hh{0.5}
    \foreach \b/\d [count=\level from 0] in {0/#1, 0/#2, #3/#4}{
      \pgfmathsetmacro{\z}{\hh * -\level}
      \pgfmathsetmacro{\zz}{\hh * -(\level+1)}
      \draw[fill=gray!75, draw=gray!75] (\b,\z) rectangle (\d,\zz);
    }
    \foreach \b/\d [count=\level from 4] in {#3/#4, #6/#5, #7/#5} {
      \pgfmathsetmacro{\z}{\hh * -\level}
      \pgfmathsetmacro{\zz}{\hh * -(\level-1)}
      \draw[fill=gray!75, draw=gray!75] (\b,\z) rectangle (\d,\zz);
    }
    \pgfmathsetmacro{\z}{\hh * -6}
    \pgfmathsetmacro{\zz}{\hh * -7}
  }
  \begin{center}
  \begin{tikzpicture}[scale=0.5]
    \tikzset{dotstyle/.style={circle,fill=black,inner sep=0pt,minimum size=3pt}}
    \begin{scope}[shift={(0,0)},scale=1]
      \drawprofile{2}{3}{1.5}{4}{6}{3}{4}{$A$}
    \end{scope}
    \draw[-stealth,red!75, thick] (6, -4) -- (6,-3)node[anchor=west,yshift=2pt]{$r_A$};
    \draw[-stealth,red!75, thick] (6, -4)node[dotstyle,fill=red!75]{}-- (7,-4) node[anchor=south,xshift=5pt]{$l_B$};
    \node [red,anchor=north east] at(6,-4){$x$};
    \begin{scope}[shift={(7,-4)},scale=1]
      \drawprofile{2}{3}{1.5}{4}{6}{3}{4}{$B$}
    \end{scope}
  \end{tikzpicture}
  \caption{Projection of the concatenation operation}\label{fig:concat2d}
  \end{center}
\end{figure}

The concatenation lemma~\cite[Lemma~6.1]{buchet2020every} generalizes straightforwardly to the higher dimensional case, as below.

\begin{lemma}\label{lem:concatenation}
    Given $A$ and $B$ two $n$D candy modules, $A\circ B$ is an $n$D candy module.
\end{lemma}
\begin{proof}
    The proof follows the exact same 
    line of reasoning
    as the $2$D case.
    
    Below, by $A$ and $B$ we mean their copies inside $A \circ B$.
    There are no paths from any vertex in the support of
    $A$ to any vertex in the support of $B$ due to the positions chosen for the two inside $A \circ B$.
    Similarly, there are no paths from $B$ to $A$.
    Furthermore, there are no paths from $A$ or $B$ to the vertex $x$.
    Hence no commutativity relations need to be checked, other than obvious $0=0$ ones and the ones internal to $A$ and $B$.
    Thus $A\circ B$ is indeed a well-defined $n$D persistence module.

    The upper left corner of 
    $A\circ B$
    coincides with the upper left corner of $A$, while
    the lower right corner of $A\circ B$ coincides with that of $B$.
    By hypothesis, both are one dimensional vector spaces.

    Finally, any endomorphism of $A\circ B$ 
    restricts to an endomorphism of $A$, an endomorphism of $B$, and of $K$ (at vertex $x$).
    Since $A$, $B$, and $K$ each have endomorphism ring isomorphic to $K$, all three of these restrictions are each uniquely determined by a scalar in $K$. As they cover the support of $A\circ B$, they are enough to uniquely determine the global endomorphism.
    Moreover, since the two arrows starting from $x$ in the support of $A\circ B$ support isomorphisms, the corresponding commutativity relations for endomorphisms 
    ensure that the three scalars are fully determined by the choice of one of them.
    Therefore the endomorphism of $A \circ B$ is isomorphic to $K$, and $A\circ B$ is indeed a candy module.
\end{proof}

\begin{theorem}
  \label{thm:lotr}
  Let $K$ be a countable field. There exists a pointwise finite dimensional indecomposable $(n+1)$D persistence module $M$ with coefficients in $K$ such that any finite dimensional $n$D persistence module $V$ with coefficients in $K$ is a hyperplane restriction of $M$.
\end{theorem}

\begin{proof}
    Given a countable family of finite dimensional (i.e.~finite total dimension) $n$D persistence modules , we can build an infinite $(n+1)$D persistence module containing all of them as hyperplane restrictions. 
    First, for each $n$D persistence module $M$ with finite total dimension in the family, build an $(n+1)$D candy module containing it as a hyperplane restriction, using the construction for Theorem~\ref{thm:nD}, and its dual, as in \eqref{eq:candywrap}.
    Then using the concatenation operation repeatedly, we build the infinite $(n+1)D$ persistence module containing all of them as hyperplane restrictions.
    
    First we claim that the resulting persistence nodule, albeit potentially non-finite, is indecomposable.
    Note that Lemma~\ref{lem:locality} does not directly apply here.
    Let us assume that $M$ is obtained by stringing a countable number of candy modules and that $M$ is decomposable into two non-zero modules $M_1$ and $M_2$.
    Then there exist vertices $(\vv a, i)$ such that $M_1(\vv a, i)\neq0$ and $(\vv b, j)$ such that $M_2(\vv b, j)\neq 0$.
    Without loss of generality, assume that $i\leq j$.
    Let $m_1-1$ be the height of the first extra vertex $x$ added in the 
    concatenation operation that lies strictly below $i$.
    If no such vertex exists, then there exists a lowest layer in the support of $M$ and let $m_1$ be the height of this layer.
    Similarly, let $m_2$ be the height of the first extra vertex $x$ that lies strictly above $j$ or the height of the highest layer of the support of $M$.
    Let $G$ be the commutative grid comprised of all layers with heights between $m_1$ and $m_2$, inclusive.
    Note that $M|_G$ is a finite candy module and therefore is indecomposable by Lemma~\ref{lem:concatenation}.
    However, $M|_G = M_1|_G \oplus M_2|_G$, and both $M_1|_G$ and $M_2|_G$ are non-zero so $M|_G$ is decomposable.
    This contradiction implies that $M$ must be indecomposable.

    It remains to show that the set of isomorphism classes of finite dimensional $n$D persistence modules with coefficients in $K$ is countable.
    First, let us fix the total dimension $d$.
    Then the number of possible dimension vectors of $n$D persistence modules with total dimension $d$ is countable.
    For each of these dimension vectors, a persistence module is defined up to isomorphism by the entries of the matrices representing its non-zero internal linear maps.
    For each persistence module under consideration, the number of its non-zero internal linear maps is finite, and the values each entry can take comes from $K$, a countable set.
    
    Thus, up to isomorphism, there is only a countable number of persistence modules with the given dimension vector.
    By taking the union over all possible total dimensions 
    $d$,
    we conclude that the set of isomorphism classes of $n$D persistence modules with finite total dimension is countable, and we can build the huge indecomposable $(n+1)$D persistence module that contains them all as hyperplane restrictions.
\end{proof}

Note that for support grids large enough (and with $n \geq 2$), the assumption that $K$ is countable is both necessary and sufficient for the number of isomorphism classes of finite total dimension $n$D persistence modules over $K$ to be countable.
Indeed, as soon as the commutative grid is representation infinite, we can build at least one indecomposable module for each element of $K$ such that no two of them are isomorphic (see construction in~\cite{buchet_et_al:socg}). This is possible as soon as the grid is of size at least $2$ by $5$ or $3$ by $3$.

On the contrary, in dimension $n=1$, the assumption on $K$ is not needed as every indecomposable (i.e. interval) is uniquely defined by its support, up to isomorphism. There are only a countable number of intervals and every $1$D persistence module with total finite dimension is isomorphic to a direct sum of a finite number of intervals. Hence there is only a countable number of them \cite[Lemma~6.2]{buchet2020every}.

\section{Minimality}
Here, we revisit the case $n=1$.
The construction presented in~\cite{buchet2020every} is not minimal in the size of the grid.
While there are several ways to define minimality for realizing a $1$D persistence module as a line restriction of an indecomposable $2$D persistence module, we consider the number of $1$D persistence modules that need to be stacked, in other words the number of layers of the indecomposable $2$D persistence module. We start with the following Lemma.

\begin{lemma}\label{lem:separator}
Let $M$ be a persistence module over the $m\times 2$ commutative grid such that there exist $\vv x\leq\vv y\leq\vv z$ with $M(\vv x)\neq 0$, $M(\vv y)=0$ and $M(\vv z)\neq 0$.
Then $M$ is decomposable.
\end{lemma}

\begin{proof}
We denote $L$ the restriction of $M$ to the first row and $U$ its restriction to the upper row.
Both $L$ and $U$ are $1$D persistence modules and can be decomposed into direct sums of intervals.
We separate those intervals into three groups for each row and show that they allow us to find a non-trivial decomposition of $M$.
We treat the case where $\vv y =(y,1)$ is located on the lower row (height $1$) below. The case of $\vv y$ on the upper row follows by a dual construction.

Since $M(\vv y) = 0$, no intervals of $L$ can be nonzero at $y$. Thus, the intervals constituting $L$ can be separated into two groups $L_1$ and $L_3$ depending on whether they are located completely to the left or to the right of $y$.
Denoting $L=\bigoplus\intv[b_i,d_i]$, we let $L_1$ be the direct sum of all intervals in $L$ with $d_i<y$ and $L_3$ the direct sum of all intervals with $b_i>y$.
Additionally, we fix $L_2=0$ and obtain $L = L_1\oplus L_2\oplus L_3$.
For the upper row, we separate the intervals depending on their death time $d_i$.
$U_1$ will be the direct sum of the intervals with $d_i<y$, $U_2$ the direct of intervals ending at $y$,
and $U_3$ the direct sum of intervals ending at $d_i>y$.
We get $U=U_1\oplus U_2\oplus U_3$.

Note that in the dual case with $\vv y$ on the upper row, we have instead $U_1$ all intervals of $U$ to the left of $y$, $U_2 = 0$, $U_3$ all intervals to the right of $y$, and separate the intervals of $L$ by comparing their birth times with $y$.

In either case, the crucial point is that there exist no non-zero homomorphisms
from $L_i$ to $U_j$ when $i\neq j$.
Using matrix notation, there exist two matrices $A$ and $B$ such that:
\begin{equation}
      M:\begin{tikzcd}[ampersand replacement=\&, column sep=large]
        L_1\bigoplus L_2 \bigoplus L_3\rar{\smat{A & 0 & 0 \\ 0 & 0 & 0\\ 0 & 0 & B}} \& U_1\bigoplus U_2\bigoplus U_3
      \end{tikzcd}
\end{equation}
This immediately decomposes into:
\begin{equation}
  \label{eq:twogriddecompo}
  M = 
  \left(
      \begin{tikzcd}[ampersand replacement=\&, column sep=large]
        L_1\rar{A} \&U_1
    \end{tikzcd}
    \right)
    \hspace{1em}\bigoplus\hspace{1em}
    \left(
    \begin{tikzcd}[ampersand replacement=\&, column sep=large]
        L_2\rar{0} \& U_2
        \end{tikzcd}
    \right)
    \hspace{1em}\bigoplus\hspace{1em} 
    \left(
    \begin{tikzcd}[ampersand replacement=\&, column sep=large]
        L_3\rar{B} \& U_3
        \end{tikzcd}
    \right)
\end{equation}

Since $M(\vv x) \neq 0$ and $\vv x \leq \vv y$, at least one interval of $L_1$
(one interval of $L_1$, $U_1$, or $L_2$ in the dual case) must be nonzero.
Similarly, since
$M(\vv z) \neq 0$ and $\vv y \leq \vv z$, at least one interval of $L_2$, $U_2$, or $U_3$ (one interval of $U_3$ in the dual case) must be nonzero.
In any case, the decomposition of $M$ in Equation~\eqref{eq:twogriddecompo} is non-trivial, and thus  $M$ is decomposable.
\end{proof}

  A simpler way to state the result of Lemma~\ref{lem:separator} is to say that the support of any indecomposable persistence module on the $m\times 2$ commutative grid is convex.
  Note that this does not hold in general for the $2$D commutative grid. For example, there exist indecomposables whose supports contain holes \cite[Section~5]{buchet2020every}.

\begin{theorem}
\label{thm:minimal}
Any $1$D persistence module can be realized as the line restriction of some indecomposable $2$D persistence module by stacking three $1$D persistence modules, and there exists a $1$D persistence module for which stacking two $1$D persistence modules is not sufficient.
\end{theorem}

\begin{proof}
The primal construction~\cite{buchet2020every} uses four layers as below:
\begin{equation}
  \label{eq:onedimS}
  S:
  \begin{tikzcd}[ampersand replacement=\&, column sep=large]
    I_V \rar{\smat{1\\\vdots\\1}} \&
    \widebar{V} \rar{\smat{1 & & \\ & \ddots & \\ & & 1}} \&
    V' \rar{\smat{1 & & \\ & \ddots & \\ & & 1}} \&
    V,
  \end{tikzcd}
\end{equation}
with the key condition that there are no non-zero morphisms between any two distinct intervals in the vertical collection $\widebar{V}$ (\cite[Lemma~4.2]{buchet2020every}).
Using a weaker condition than verticality for $\widebar{V}$, we can essentially skip $V'$ and reduce the number of layers needed by one.

Given a module
$V=\bigoplus_{i=1}^l \intv[b_i,d_i]$,
we choose a set of pairwise distinct values
$b'_1,\dots,b'_l$
such that
$b_i\leq b'_i \leq d_i$ for every $i$.
This is not possible in general as we would need more distinct values than can exist between $b_i$ and $d_i$.
However, following the definition of a line restriction being able to skip points (see the comment after Definition~\ref{defn:hyperplane_restriction}), we are free to add as many intermediate points as necessary, which the line restriction then skips.
Special care must be taken for multiple copies of the same simple summand $\intv[b_i,d_i]$ with $b_i = d_i$. For those, we artificially inflate the simple summands as $\intv[b_i-\epsilon, b_i+\epsilon]$ for some $\epsilon < 0.5$, where $b_i-\epsilon$ and $b_i+\epsilon$ can be taken as intermediate indices as above. The line restriction does not ``see'' this inflation, and is able to correctly recover $V$.

By a reordering of indices, we can assume that $b'_1<\dots<b'_l$.
We then set the values $d'_i$ for every $i$ such that $d'_i\geq d_i$ and $d'_1>\dots>d'_l$. Thus, we have
\begin{equation}
\label{ineq:bdorder1}
  b'_1 < \dots < b'_l \leq d'_l < \dots < d'_1
\end{equation}
and so
$
  b'_i < b'_j \leq d'_j < d'_i
$
whenever $i < j$. Using a proof similar to the one for \cite[Lemma~4.2]{buchet2020every}, we can prove that the matrix form of any endomorphism of the module
\[
  V''=\bigoplus_{i=1}^l \intv[b'_i,d'_i]
\]
is diagonal. Define
$
  I'_V=\intv[\max_i b'_i,\max_i d'_i] 
$.
We then claim that
\begin{equation}
  \label{eq:onedimShort}
  \hat{S}:
  \begin{tikzcd}[ampersand replacement=\&, column sep=large]
    I'_V \rar{\smat{1\\\vdots\\1}} \&
    V'' \rar{\smat{1 & & \\ & \ddots & \\ & & 1}} \&
    V
  \end{tikzcd},
\end{equation}
which can be viewed as a $2$D persistence module via stacking,
is indecomposable. We again remind the reader that  the ``$1$'''s are not identity maps.

By construction, $\hat S$ is well-defined.
Unlike the construction in \cite{buchet2020every} which provides $S$ with $\End{S}\cong K$, the endomorphism ring $\End{\hat S}$ is not necessarily isomorphic to $K$.
However, we can show that $\End{\hat S}$ is local.

To do this, we first consider the relation $\preccurlyeq$ on intervals defined by $\intv[a,b]\preccurlyeq\intv[c,d]$ if and only if there exists a non-zero morphism from $\intv[a,b]$ to $\intv[c,d]$, i.e. $c\leq a \leq d \leq b$ (\cite[Part~3~of~Lemma~3.2]{buchet2020every}).
This relation is reflexive and antisymmetric as $\intv[a,b]\preccurlyeq\intv[c,d]$ and $\intv[c,d]\preccurlyeq\intv[a,b]$ imply $a=c$ and $b=d$.
Second, given three distinct intervals $\intv[a,b]$, $\intv[c,d]$ and $\intv[e,f]$, if $\intv[a,b]\preccurlyeq\intv[c,d]$, and $\intv[c,d]\preccurlyeq\intv[e,f]$, then the relation $\intv[e,f]\preccurlyeq\intv[a,b]$ is impossible.
Indeed if this relation were satisfied then $a=e$ and $b=f$, but the two intervals are distinct by assumption.

We once more reorder the indices in $V$ such that identical intervals are grouped together and that
distinct intervals with
$\intv[b_i,d_i] \preccurlyeq \intv[b_j,d_j]$
implies $i < j$.
We similarly reorder the indices for the intervals composing $V''$ such that the map $V''\rightarrow V$ has the diagonal matrix form
$
  \smat{1 & & \\ & \ddots & \\ & & 1}.
$
Note that Inequality~\eqref{ineq:bdorder1} may no longer be satisfied after this reordering. This is not a problem however, as the matrix form of any endormorphism of $V''$ is still diagonal even after reordering.

Let $\phi\in\End{\hat S}$ and
consider its components $\phi_1:I_V\rightarrow I_V$, $\phi_2:V''\rightarrow V''$ and $\phi_3:V\rightarrow V$.
There exists $\alpha\in K$ such that $\phi_1$ is $\alpha$ times the identity of $I_V$. Furthermore, by commutativity relation and the fact that the matrix form of $\phi_2$ is diagonal, the matrix form of $\phi_2$ is $\smat{\alpha & &\\ &\ddots &\\ & &\alpha}$.

A key ingredient in the proof \cite{buchet2020every} of indecomposability of $S$ in \eqref{eq:onedimS}, the propagation of nonzeros lemma (\cite[Lemma~4.3]{buchet2020every}), does not apply in this setting for $\hat S$.
Instead, let us carefully consider the form of $\phi_3$,
denoting the entry at row $j$ column $i$ by ${\phi_3}_{j,i}$.
If $i > j$ and $\intv[b_i,d_i]\neq \intv[b_j,d_j]$, then there is no non-zero morphism from $\intv[b_i,d_i]$ to $\intv[b_j,d_j]$.
Hence ${\phi_3}_{j,i}=0$ for $i > j$ and $\intv[b_i,d_i]\neq \intv[b_j,d_j]$. Thus $\phi_3$ is block lower triangular.

Each of group of identical intervals in $V$ corresponds to a block along the diagonal of the matrix of $\phi_3$.
As all the intervals involved in a block on the diagonal are identical, the canonical morphisms between them are identity morphisms.
Using the commutativity relation between $\phi_3$ and $\phi_2$ with matrix form $\smat{\alpha & &\\ &\ddots &\\ & &\alpha}$,
each block on the diagonal must be $\alpha$ times the identity of the concerned intervals.

Note that the commutativity relation does not always guarantee that the off-diagonal blocks are zero. This stems from the fact that the composition of two non-zero canonical morphisms between intervals may be zero, and thus the relation only requires $0=0$, which always holds anyway with no restriction on the off-diagonal block. In contrast, along the diagonal, we have identity maps as canonical homomorphisms, which never goes to zero when composed with non-zeros.

To summarize, the matrix form of $\phi_3$ is lower triangular with all diagonal terms equal to $\alpha$ times a corresponding identity morphism. That is, each $\phi \in \End{\hat S}$ determines (but is not uniquely determined by) an $\alpha \in K$. The assignment $c(\phi) := \alpha$ determines an algebra homomorphism $c:\End(\hat S)\rightarrow K$.

Next, we check that the endomorphism $\phi$ is invertible if and only if $\alpha = c(\phi) \neq 0$. Note that we cannot rely on the determinant to ensure invertibility as the matrix entries are morphisms and not numbers. If $\alpha=0$, then $\phi$ is clearly not invertible, since $1 = \phi\psi$ implies that $1 = c(1) = c(\phi\psi) = c(\phi)c(\psi) = 0$, a contradiction.

On the other hand, suppose that $\alpha\neq 0$. We can find the inverse of $\phi$ by carefully using the classical Gaussian elimination algorithm on $\phi_3$. First, we use row operations to construct a left inverse for $\phi_3$.
As the entries are morphisms possibly having different domains and co-domains, each row operation must include a composition with the canonical morphism between the corresponding spaces.
To ensure that a complete elimination of off-diagonal entries is possible, this composition should never be $0$ when an entry must be eliminated (a non-zero off-diagonal).
To see this, note that a non-zero off-diagonal entry corresponds to a multiple of a non-zero canonical morphism. On the other hand, the pivot entries are on the diagonal with identity morphisms as their canonical morphisms, and thus have non-zero composition with any other non-zero morphism. This permits the elimination of the non-zero off-diagonal entries. Similarly, we construct a right inverse for $\phi_3$ using Gaussian elimination with column operations. Thus, $\phi_3$ (and therefore also $\phi$) is invertible, if $c(\phi) \neq 0$.

Using the above fact, let us check that $\End{\hat S}$ is local.
Suppose that an endomorphism $\phi$ is not invertible. Then $c(\phi)=0$, from which we conclude that $1-\phi$ is invertible as its matrix form is lower triangular with diagonal terms all equal to $c(1-\phi) = 1-c(\phi) = 1 \neq 0$.
Therefore $\End{\hat S}$ is local, and $\hat S$ is indecomposable.

This construction has the minimal number of rows.
Consider the module $V=\intv[a,b]\bigoplus\intv[c,d]$ with $a< b+1<c\leq d$.
Assume that $M$ contains $V$ as a line restriction, by a line $L$. Let $L$ send $b$ to $\vv x$, $b+1$ to $\vv y$, and $c$ to $\vv z$. Since $L$ is injective on objects, the vertices $\vv x$, $\vv y$, and $\vv z$ are all distinct.
We have 
$M(\vv x) = ML(b) \cong V(b) = K$, 
$M(\vv y) = ML(b+1) \cong V(b+1) = 0$, 
$M(\vv z) = ML(c) \cong V(c) = K$, 
and $\vv x\leq \vv y\leq \vv z$. By Lemma~\ref{lem:separator}, $M$ is decomposable.
\end{proof}

The construction of the above Theorem can be easily generalized to finite rectangle-decomposable modules in $n$D. However, we were unable to show that three layers is the minimum required in this case.
\begin{theorem}
  \label{thm:minimal_rect}
  Any $n$D finite rectangle-decomposable persistence module can be realized as the hyperplane restriction of some indecomposable $(n+1)$D persistence module by stacking three $n$D persistence modules, and any $n$D persistence module, by stacking four layers.
\end{theorem}

\begin{proof}[Sketch of Proof]
Without loss of generality, we let $V = \bigoplus_{i=1}^m \intv[\vv{b_i},\vv{d_i}]$, a finite rectangle-decomposable $n$D persistence module.
We choose $\vv{b'_1}, \vv{b'_2}, \hdots, \vv{b'_m}$ with pairwise \emph{distinct first coordinates} such that
$\vv{b_i} \leq \vv{b'_i} \leq \vv{d_i}$.
As before, we take advantage of the definition of hyperplane restrictions to find enough distinct points to place between the births and deaths, with a similar inflation for simple summands. Since the first coordinates are distinct, without loss of generality we can rearrange the terms to have
$\vv{b'_1} \lessf \vv{b'_2} \lessf \hdots \lessf \vv{b'_m}$, where $\vv{x} \lessf \vv{y}$ means that the first coordinate of $\vv{x}$ is strictly less than that of $\vv{y}$.

We then choose $\vv{d'_i}$ with $\vv{d_i} \leq \vv{d'_i}$ so that 
$\vv{d'_m} \lessf \hdots \lessf \vv{d'_1}$. Overall, we have
\begin{equation}
\label{eq:series}
\vv{b'_1} \lessf \hdots \lessf \vv{b'_m}
\leq
\vv{d'_m} \lessf \hdots \lessf \vv{d'_1}
\end{equation}
and
\begin{equation}
\label{eq:compa}
    \vv{b_i} \leq \vv{b'_i} \leq \vv{d_i} \leq \vv{d'_i}
\end{equation}
for each $i$.

We then construct
$
V'' = \bigoplus\limits_{i=1}^m \intv[\vv{b'_i},\vv{d'_i}]
$
together with the map 
\begin{equation}
  \begin{tikzcd}[ampersand replacement=\&, column sep=large]
    V'' \rar{\smat{1 & & \\ & \ddots & \\ & & 1}} \&
    V
  \end{tikzcd}
\end{equation}
where the existence of the  nonzero morphisms (represented by $1$'s in the matrix form) on the diagonal are the canonical morphisms guaranteed by Inequalities~\eqref{eq:compa}.
It can be shown using the Inequalities~\eqref{eq:series} that
$
\Hom(\intv[\vv{b'_i},\vv{d'_i}], \intv[\vv{b'_j},\vv{d'_j}]) = 0
$
for $i \neq j$. Furthermore,
$
\Hom(\intv[\vv{b'_i},\vv{d'_i}], \intv[\vv{b'_i},\vv{d'_i}]) \cong K
$
for each $i$. Thus, the matrix forms of the endomorphisms of $V''$ are diagonal.

Finally, we define 
$
    I_V \defeq \intv\left[\max\limits_i\left(\vv{b'_i}\right), \max\limits_j\left(\vv{d'_j}\right)\right]
$
and form 
\begin{equation}
  \hat S:
  \begin{tikzcd}[ampersand replacement=\&, column sep=large]
    I'_V \rar{\smat{1\\\vdots\\1}} \&
    V'' \rar{\smat{1 & & \\ & \ddots & \\ & & 1}} \&
    V
  \end{tikzcd},
\end{equation}
an object of $\fun{\Af{3}}{\rep(\mathbb{Z}^n,\leq)}$, a $\rep(\mathbb{Z}^n,\leq)$-valued representation of $\Af{3}$. The proof that $\hat S$ is indecomposable is similar to that for the $1$D case given in Theorem~\ref{thm:minimal}. In particular, it can be shown that $\hat S$ is local.

For the second part of the theorem, we let $V$ be an $n$D persistence module with finite support.
Then, as in the arguments used in the proof for Theorem~\ref{thm:nD}, there exists a surjection $p: R\twoheadrightarrow V$ from $R$ a finite rectangle-decomposable module.

Using the three layer construction above for the finite rectangle-decomposable module $R$ and appending $p: R\twoheadrightarrow V$, we obtain
\[
    \hat S':
    \begin{tikzcd}[ampersand replacement=\&, column sep=large]
      I'_R \rar{\smat{1\\\vdots\\1}} \&
      R'' \rar{\smat{1 & & \\ & \ddots & \\ & & 1}} \&
      R \arrow[two heads]{r}{p}
      \& V.
    \end{tikzcd}
  \]
Let us show that $\hat S'$ is indecomposable. Consider $\phi \in \End{\hat S'}$. The restriction $\phi|_{\hat S}$ of $\phi$ to the first three terms gives an endomorphism of 
$
  \hat S:
  \begin{tikzcd}[ampersand replacement=\&, column sep=large]
    I'_R \rar \&
    R'' \rar \&
    R
  \end{tikzcd}
$. Since $\End{\hat S}$ is local, either $\phi|_{\hat S}$ or $1-\phi|_{\hat S}$ is invertible. Suppose that $\phi|_{\hat S}$ is invertible, and so $\phi|_R$ is invertible. The commutative relation on $\phi$ requires that $(\phi|_V) p = p (\phi|_R)$ so that $(\phi|_V) p (\phi|_R)^{-1} = p$. This shows that $\phi|_V:V\rightarrow V$ is a surjection. Since $V$ has finite total dimension, this shows that $\phi|_V$ is an isomorphism, and thus $\phi$ is invertible.
Similarly, $1-\phi|_{\hat S}$ invertible implies that $1-\phi$ is invertible.
Therefore $\End{\hat S'}$ is local, and $\hat S'$ is indecomposable.
\end{proof}

\section{Discussion}

The constructions provided in this paper 
illustrates the jumps in complexity when going from $n$D to $(n+1)$D persistence.
Namely, the whole complexity of $n$D persistence modules is fully contained in the building blocks one dimension higher, the
indecomposable $(n+1)$ persistence modules.
We hope that this point of view will lead to deeper insights into the algebraic difficulties of multidimensional persistence.

We also provided additional insight into a restriction with possibly reduced complexity, from the point of view of hyperplane restrictions.
The consideration of indecomposable $(n+1)$D persistence modules obtained by the stacking of exactly two $n$D persistence modules might be strictly less complex compared to the the general case, in terms of indecomposables containing all persistence modules one dimensional lower as hyperplane restrictions.
We showed that this is indeed the case in the $(n+1)=2$ case, where it corresponds to commutative ladders which have been studied in more details~\cite{escolar2016persistence}.
In higher dimensions, we note that even when restricted to two layers, it is still of wild representation type (a direct consequence of~\cite[Theorem~2.5]{leszczynski1994representation},~\cite[Theorem~1.3]{bauer2020cotorsion}). From that point of view, the restriction does not reduce the complexity. It may be worthwhile to explore the connections between representation type and hyperplane restrictions of indecomposables.

\bibliographystyle{plain}
\bibliography{refs}

\end{document}
